\documentclass[11pt]{amsart}
\usepackage{amsmath, amscd, amssymb}
\usepackage[enableskew,vcentermath]{youngtab}
\usepackage{ytableau}
\usepackage[mathscr]{eucal}
\usepackage[normalem]{ulem} 
\usepackage{soul} 
\usepackage[all,cmtip]{xy}
\usepackage{tikz-cd}
\usepackage{color} 
\setulcolor{red}
\sethlcolor{yellow}
\setstcolor{blue}
\usepackage{xfrac} 

\usepackage{geometry}
\usepackage{amsfonts}
\usepackage[colorlinks,linkcolor=blue,citecolor=blue, pdfstartview=FitH]{hyperref}
\usepackage{graphicx}
\usepackage{cite}

\numberwithin{equation}{section}

\theoremstyle{definition}
\newtheorem{thm}{Theorem}[section]
\newtheorem{theorem}[thm]{Theorem}

\newtheorem{lemma}[thm]{Lemma}
\newtheorem{corollary}[thm]{Corollary}
\newtheorem{proposition}[thm]{Proposition}

\newtheorem{remark}[thm]{Remark}

\newtheorem{definition}[thm]{Definition}

\newtheorem{example}[thm]{Example}

\newtheorem{defn-thm}[thm]{Definition-Theorem}

\newenvironment{observe}{\noindent\textcolor{blue}{\textit{Observation}}.}{\hfill \textcolor{blue}{$\blacktriangleleft$}\par}
\newtheorem*{theorem*}{Theorem}
\newtheorem*{proposition*}{Proposition}

\usepackage{todonotes}
\definecolor{pistachio}{rgb}{0.58, 0.77, 0.45}
\definecolor{eggshell}{rgb}{0.94, 0.92, 0.84}

\newcommand{\cN}{{\mathcal N}}

\newcommand{\cP}{{\mathcal P}}

\newcommand{\g}{{\mathfrak g}}

\newcommand{\gp}{\mathfrak{p}}

\newcommand{\CC}{{\mathbb C}}

\newcommand{\Sp}{\operatorname{Sp}}
\newcommand{\SO}{\operatorname{SO}}

\newcommand{\id}{{\operatorname{id}}}

\newcommand{\btheorem}{\begin{theorem}}
	\newcommand{\etheorem}{\end{theorem}}
\newcommand{\bproposition}{\begin{proposition}}
	\newcommand{\eproposition}{\end{proposition}}
\newcommand{\bdefinition}{\begin{definition}}
	\newcommand{\edefinition}{\end{definition}}
\newcommand{\bcorollary}{\begin{corollary}}
	\newcommand{\ecorollary}{\end{corollary}}
\newcommand{\bproof}{\begin{proof}}
	\newcommand{\eproof}{\end{proof}}
\newcommand{\bremark}{\begin{remark}}
	\newcommand{\eremark}{\end{remark}}
\newcommand{\eexample}{\end{example}}
\newcommand{\bexample}{\begin{example}}
\newcommand{\elemma}{\end{lemma}}
\newcommand{\blemma}{\begin{lemma}}
\newcommand{\bobserve}{\begin{observe}}
	\newcommand{\eobserve}{\end{observe}}

\renewcommand{\bar}{\overline}

\renewcommand{\phi}{\varphi}

\newcommand{\ee}{\end{eqnarray*}}
\newcommand{\be}{\begin{eqnarray*}}

\newcommand{\beq}{\begin{equation}}
	\newcommand{\eeq}{\end{equation}}

\newcommand{\bd}{\begin{enumerate}}
	\newcommand{\ed}{\end{enumerate}}
\newcommand{\bti}{\begin{tikzcd}}
	\newcommand{\eti}{\end{tikzcd}}

\renewcommand{\tilde}{\widetilde}

\renewcommand{\sf}[1]{\textsf{#1}}
\renewcommand{\bf}[1]{\mathbf{#1}}

\title[Spaltenstein Varieties Associated with Pseudo-Polarizations]{Spaltenstein Varieties Associated with Pseudo-Polarizations}

\author{Xueqing Wen}
\address{Mathematical Science Research Center, Chongqing University of Technology, No.69, Hongguang Avenue, Banan District, Chongqing, 400054, China}
\email{wenxq@cqut.edu.cn}

\author{Yaoxiong Wen}
\address{School of Mathematical Science, Zhejiang University, 866 Yuhangtang Rd, Hangzhou, 310058, Zhejiang, China}
\email{y.x.wen.math@gmail.com}
\date{}

\begin{document}

\begin{abstract}
    We introduce minimal Richardson orbits and pseudo-polarizations for nilpotent orbits in classical Lie algebras of types~B,~C, and~D. For any nilpotent orbit, we classify all minimal Richardson orbits containing it and thereby determine the associated pseudo-polarizations. We prove that the corresponding Spaltenstein varieties are smooth and pure dimensional, with iterated orthogonal/isotropic Grassmannian fibrations. As an application, we extend the seesaw property and duality of Fu--Ruan--Wen from Richardson orbits to all special orbits in types~B and~C.
\end{abstract}

\maketitle

\section{Introduction}

Let $G$ be a reductive group with Lie algebra $\g$, and let $\cN \subset \g$ denote the nilpotent cone.  
For a nilpotent element $e \in \g$, its adjoint orbit $G \cdot e =: \bf{O}_e$ is called a \emph{nilpotent orbit}.

A nilpotent orbit $\bf{O}_R$ is called \emph{Richardson} if there exists a parabolic subgroup $P \subset G$ such that the image of the moment map (or \emph{generalized Springer map})
\begin{align*}
    \mu_P: T^*(G/P) \longrightarrow \g
\end{align*}
coincides with the closure $\overline{\bf{O}}_R$. In this case, $P$ is called a \emph{polarization} of $\bf{O}_R$. A Richardson orbit may admit several distinct polarizations. In general, $\mu_P$ is generically finite; if it is birational, then it is crepant. Moreover, any crepant resolution of the normalization of a nilpotent orbit closure arises in this way \cite{Fu03}. In type~A, every nilpotent orbit is Richardson, and the corresponding generalized Springer maps $\mu_P$ are always crepant.

In \cite{FRW24}, Fu, Ruan, and the second-named author studied Richardson orbits in types~B and~C and discovered a striking link between Springer duality and Langlands duality.  
More precisely, if $\bf{O}_{B,R}$ and $\bf{O}_{C,R}$ are Springer-dual Richardson orbits with dual polarizations $(P_B,P_C)$ (i.e.\ the Levi factors of $P_B$ and $P_C$ are Langlands dual), then one has the diagram
\begin{align}
\begin{split}\label{intro.diagram} 
\xymatrix{ 
 T^*( \SO_{2n+1}/P_B ) \ar[rd]  \ar[dd]_{\mu_{P_B}} 
 & \stackrel{\text{Langlands dual}}{\leftrightsquigarrow} 
 & T^*( \Sp_{2n} / P_C ) \ar[ld] \ar[dd]^{\mu_{P_C}}  \\   
  & (Z_{P_B}, Z_{P_C}) \ar[ld] \ar[rd] &  \\
 \overline{\bf{O}}_{B,R} 
 & \stackrel{\text{Springer dual}}{\leftrightsquigarrow} 
 & \overline{\bf{O}}_{C,R} 
} 
\end{split}
\end{align}
together with the \emph{seesaw identity}
\begin{align*}
    \deg \mu_{P_B} \cdot \deg \mu_{P_C} 
    = \#\bar{A}(\bf{O}_{B,R}) 
    = \#\bar{A}(\bf{O}_{C,R}),
\end{align*}
where $\bar{A}(-)$ denotes Lusztig’s canonical quotient.  
Moreover, the Stein factorization pair $(Z_{P_B},Z_{P_C})$ forms a mirror pair in the sense that they share the same stringy Hodge numbers \cite{FRW24}.

While these results reveal a deep interaction between Langlands duality and Springer duality, they apply only to Richardson orbits. However, Springer duality is defined for the broader class of \emph{special} nilpotent orbits, many of which are not Richardson and hence admit no polarization in the usual sense. This prevents a direct extension of the above seesaw picture beyond the Richardson setting.

The primary motivation of this paper is to overcome this obstruction by introducing a generalized notion of polarization that applies to arbitrary nilpotent orbits.

\begin{definition}
Let $\bf{O}_e$ be a nilpotent orbit of type~B,~C, or~D.  
A \emph{minimal Richardson orbit} $\bf{O}_R$ for $\bf{O}_e$ is a Richardson orbit of the same type such that 
(i) $\overline{\bf{O}}_R \supset \bf{O}_e$, and 
(ii) $\bf{O}_R$ is minimal with respect to the closure order among all such Richardson orbits.

A \emph{pseudo-polarization} of $\bf{O}_e$ is any polarization $P$ of a minimal Richardson orbit $\bf{O}_R$ associated with $\bf{O}_e$.
\end{definition}

In Propositions~\ref{prop.mim_R_B}, \ref{prop.mim_R_C}, and \ref{prop.mim_R_D}, we classify all minimal Richardson orbits for a given nilpotent orbit $\bf{O}_e$, thereby all pseudo-polarizations of $\bf{O}_e$ can be determined, see Remark~\ref{rem.pseudo-polarization}.

Having fixed such a pseudo-polarization $P$, we then study the geometry of the corresponding fiber of the generalized Springer map. Concretely, we consider the restricted diagram
\begin{align}\label{intro.diagram2}
\begin{split}
    \xymatrix{
     & T^*(G/P) \ar[d]^{\mu_P} \\
     \bf{O}_e \ar@{^{(}->}[r] & \overline{\bf{O}}_R 
    }
\end{split}
\end{align}
For $e \in \bf{O}_e$, the fiber $\mu_P^{-1}(e)$ is called the \emph{Spaltenstein variety} associated with $(G,P,e)$. In general, Spaltenstein varieties are neither smooth nor irreducible. Even in cases where an affine paving exists \cite{Spa76, Spa83a, DeCLP88, DeCM22}, singularities may still occur.

In addition, unlike Springer fibers, Spaltenstein varieties need not be pure dimensional \cite[\S~11.6]{Spa82}. Nevertheless, relatively little is known in general about when pure dimensionality holds.  
Pure dimensionality holds for Springer fibers (i.e., when $P$ is a Borel subgroup) and for all Spaltenstein varieties of type~A, by work of Spaltenstein and Steinberg \cite{Ste74,Spa76,Spa77}. In addition, Li \cite{Li20} showed that for types~B,~C, and~D, pure dimensionality holds whenever the nilpotent orbit has partition $[1^{\sf w_1}, 2^{\sf w_2}, 3^{\sf w_3}, \ldots]$ with $\sf w_i \sf w_{i+1}=0$ for all $i$.

Our main geometric result is as follows.

\begin{theorem}[Theorems~\ref{Thm:Spal_fiber B}, \ref{Thm:Spal_fiber C}, \ref{Thm:Spal_fiber D}]
For any nilpotent orbit $\bf{O}_e$ of type~B,~C, or~D and any pseudo-polarization $P$, the reduced Spaltenstein variety $\mu_P^{-1}(e)^{\mathrm{red}}$ is an iterated fibration over a point, where each fiber is isomorphic to either an orthogonal Grassmannian or an isotropic Grassmannian. In particular, $\mu_P^{-1}(e)^{\mathrm{red}}$ is smooth and pure dimensional.
\end{theorem}

This result contrasts sharply with the classical situation: although the subregular orbit is itself Richardson, subregular Springer fibers are typically singular. The smoothness obtained here confirms that minimal Richardson orbits and pseudo-polarizations capture the correct geometric replacement for polarizations in the special-orbit setting.

Returning to types~B and~C, let $\mathbf{O}_B$ and $\mathbf{O}_C$ be Springer dual special orbits.  
Let $\mathbf{O}_{B,R}$ be a minimal Richardson orbit for $\mathbf{O}_B$. By Proposition~\ref{prop.Spr dual min Richardson}, its Springer dual $\mathbf{O}_{C,R}$ is a minimal Richardson orbit for $\mathbf{O}_C$. Choosing Langlands dual pseudo-polarizations $P_B$ and $P_C$, we obtain
\begin{align*}
    \xymatrix{ 
     & T^*(\SO_{2n+1}/P_{B})  \ar[d]_{\mu_{P_B}} 
     & \stackrel{\text{Langlands dual}}{\leftrightsquigarrow} 
     & T^*(\Sp_{2n} / P_{C}) \ar[d]^{\mu_{P_C}} \\     
     \mathbf{O}_{B} \ar@{^{(}->}[r] 
     & \overline{\mathbf{O}}_{B,R} 
     & \stackrel{\text{Springer dual}}{\leftrightsquigarrow} 
     & \overline{\mathbf{O}}_{C,R} 
     & \mathbf{O}_{C} \ar@{_{(}->}[l]  
    }
\end{align*}

Let $\#\mu_{P_B}^{-1}(e_B)$ and $\#\mu_{P_C}^{-1}(e_C)$ denote the numbers of connected components of the corresponding Spaltenstein varieties.

\begin{theorem}[Theorem~\ref{thm.duality}]
With the notation above:
\begin{itemize}
\item[(1)] One has the generalized seesaw identity
\begin{align*}
\#\mu_{P_B}^{-1}(e_B)\cdot \#\mu_{P_C}^{-1}(e_C)
= \#\bar{A}(\mathbf{O}_{B})
= \#\bar{A}(\mathbf{O}_{C}),
\end{align*}
where $\bar{A}(\mathbf{O})$ is Lusztig’s canonical quotient.

\item[(2)] The connected components of $\mu_{P_B}^{-1}(e_B)$ and $\mu_{P_C}^{-1}(e_C)$ have the same $E$-polynomial.
\end{itemize}
\end{theorem}

Part~(1) extends \eqref{intro.diagram} from Richardson orbits to all special orbits. Although we cannot yet construct mirror finite covers of special orbit closures, the behavior of Spaltenstein varieties provides strong geometric evidence for a broader Springer–Langlands–mirror correspondence.

Finally, we note that special orbits play a central role in modern geometric representation theory.  
In \cite{WWW24}, it was shown that two special orbits in types~B and~C are Springer dual if and only if their associated residually nilpotent Hitchin systems have isomorphic Hitchin bases;  
in \cite{WWW25}, we proved that a type~D orbit is special if and only if the generic fiber of its residually nilpotent Hitchin system is a torsor under a \emph{self-dual} abelian variety.  
The present work fits naturally into this emerging picture.

\medskip
\noindent\textbf{Organization of the paper.} In Section~\ref{Sec:preliminaries} we review partitions, special orbits, and Richardson orbits in types B/C/D. In Section~\ref{sec.mini_R} we classify minimal Richardson orbits and pseudo-polarizations. In Section~\ref{s.Spal_fiber} we prove the iterated orthogonal/isotropic Grassmannian fibration description of Spaltenstein varieties. Finally, in Section~\ref{Sec: E poly} we establish the E-polynomial duality and the generalized seesaw identity for Springer dual special orbits in types B and C.

\subsection*{Acknowledgements}
We would like to thank Professor Baohua Fu, Professor Yongbin Ruan, and Bin Wang for many helpful discussions. Y. Wen would like to thank Woonam Lim for his support during the author’s stay at Yonsei University, where this work was carried out.

\section{Preliminaries on nilpotent orbits}\label{Sec:preliminaries}

Let $G$ be a complex reductive group, and $\g$ be its Lie algebra. Taking a nilpotent element $e \in \g$, the adjoint orbit $\mathbf{O}_e:= G \cdot e$ is called a \emph{nilpotent orbit}. In the following, we will focus on types B, C, and D, i.e., $G=\SO_{2n+1}, \Sp_{2n}$, or $\SO_{2n}$. We will use ``$\equiv$'' to denote the congruence modulo 2.

Nilpotent orbits of classical types are known to be characterized by certain types of partitions. For more details, see \cite{CM93}. In this discussion, we will not distinguish between the partitions and the nilpotent orbits they label. Fix an integer $N$. We denote by $\cP(N)$ the set of partitions $\bf{d} =[d_1 \geq d_2 \cdots \geq d_r]$ of $N$.

For $\varepsilon= \pm 1$, we define:
\begin{align*}
	\mathcal{P}_{\varepsilon}(N) = \big\{ [d_1, \ldots, d_r] \in \mathcal{P}_{}(N) \ \big{|} \ \sharp \{j \ | \ d_j =i \} \equiv 0 \; \text{for} \; \forall i \; (-1)^i= \varepsilon   \big\}.
\end{align*}

It is well-known that $\cP(N)$, $\cP_{1}(2n+1)$, and $\cP_{-1}(2n)$ are in one-to-one correspondence with the nilpotent orbits in $\mathfrak{sl}_N$, $\mathfrak{so}_{2n+1}$, and $\mathfrak{sp}_{2n}$, respectively. For $\cP_{1}(2n)$, this correspondence holds except in special cases where all parts in $\bf{d}_D \in \cP_1(2n)$ are even. The corresponding orbit $\bf{O}_{\bf{d}_D}$ is called \emph{very even}. As an $\SO_{2n}$-orbit, it has two connected components: $\bf{O}_{\bf{d}_D}= \bf{O}_{\bf{d}_D}^{I} \bigsqcup \bf{O}_{\bf{d}_D}^{II}$.

The set $\cP(N)$ is partially ordered as follows:  
\begin{align}\label{partial-order}
    \bf{d} =\left[ d_1, \ldots, d_N  \right] \geq \mathbf{f}= \left[ f_1, \ldots, f_N  \right]  \Longleftrightarrow  \sum_{j=1}^k d_j \geq \sum_{j=1}^k f_j, \quad \text{for all} \quad 1 \leq k \leq N. 
\end{align}
This induces a partial order on $\mathcal{P}_{\pm 1}(N)$, which coincides with the partial order on nilpotent orbits given by the inclusion of closures (cf. \cite[Theorem 6.2.5]{CM93}). Using this partial order, the closure of nilpotent orbit $\bf{O}_{\bf{d}}$ with partition $\bf{d}$ is defined as the union of all smaller orbits 
\[
    \bar{\bf{O}}_{\bf{d}} = \bigsqcup_{\bf{f} \leq \bf{d} } \bf{O}_{\bf{f}} .
\]

For a partition $\mathbf{d}=\left[ d_1,\ldots, d_{r} \right]$ in $\cP(2n+1)$, there is a unique largest partition in $\cP_{1}(2n +1)$ dominated by $\mathbf{d}$. This partition, called the $B$--collapse of $\mathbf{d}$, is denoted by $\mathbf{d}_B$. The $B$--collapse can be computed as follows: if $\mathbf{d}$ is not already in $\cP_{1}(2n+ 1)$, then at least one of its even parts must occur with odd multiplicity. Let $q$ be the largest such part. Replace the last occurrence of $q$ in $\mathbf{d}$ by $q-1$ and the first subsequent part $r$ strictly less than $q-1$ by $r+1$. This process is called a \emph{collapse}. Repeat the process until a partition in $\cP_{1}(2n +1)$ is obtained, which is denoted by $\mathbf{d}_B$. Geometrically, let $\bf{O}_{\bf{d}}$ be a nilpotent orbit in $\mathfrak{sl}_{2n+1}$, then $\overline{\bf{O}}_{\mathbf{d}_B} = \overline{\bf{O}}_{\bf{d}} \cap \mathfrak{so}_{2n+1} $.

Similarly, there is a unique largest partition $\mathbf{d}_{D/C}$ in $\cP_{\pm 1}(2n)$ dominated by any given partition $\mathbf{d}$ of $2n$, which is called the $D/C$--collapse of $\mathbf{d}$. The procedure to compute $\mathbf{d}_{D/C}$ is similar to that of $\mathbf{d}_B$. 

To streamline later results, we introduce specific types of partitions as building blocks:
\vspace{0.2 cm}

\textbf{For type B}:
\begin{itemize}
	\item Type B1: $[\alpha, \alpha]$, where $\alpha \equiv 1$.
	\item Type B1*: $[\beta, \beta]$, where $\beta \equiv 0$.
	\item  Type B2: $[\alpha_1, \beta_1, \beta_1, \beta_2, \beta_2, \ldots, \beta_k, \beta_k, \alpha_2]$, where $\alpha_1 > \beta_1 \geq \beta_2 \geq \ldots \geq \beta_k > \alpha_2$, $\alpha_i \equiv 1$, $\beta_j \equiv 0$, and $k \geq 0$.
	\item Type B3: $[\alpha_1, \beta_1, \beta_1, \beta_2, \beta_2, \ldots, \beta_k, \beta_k ]$, where the same conditions as in Type B2.
\end{itemize}

\textbf{For type C}, we define: 
\begin{itemize}
	\item Type C1: $[\alpha, \alpha]$, where $\alpha \equiv 1$.
	\item Type C1*: $[\beta, \beta]$, where $\beta \equiv 0$.
	\item  Type C2: $[\beta_1, \alpha_1, \alpha_1, \alpha_2, \alpha_2, \ldots, \alpha_k, \alpha_k, \beta_2]$, with $\beta_1 > \alpha_1 \geq \alpha_2 \geq \ldots \geq \alpha_k > \beta_2 \geq 0$,\footnote{$\beta_2$ can be 0. If $\beta_2=0$, we omit $\beta_2$.} $\alpha_i \equiv 1$, $\beta_j \equiv 0$, and $k \geq 0$.
\end{itemize}

\textbf{For type D}:
\begin{itemize}
	\item Type D1: $[\alpha, \alpha]$, where $\alpha \equiv 1$.
	\item Type D1*: $[\beta, \beta]$, where $\beta \equiv 0$.
	\item  Type D2: $[\alpha_1, \beta_1, \beta_1, \beta_2, \beta_2, \ldots, \beta_k, \beta_k, \alpha_2]$, as in Type B2.
\end{itemize}

Then, we can uniquely rewrite the partition of the nilpotent orbit using the above building blocks:

\begin{lemma}\label{lem.building block}
	For any partition $\bf{d}$ labeling a nilpotent orbit of type B, C, or D, it can be uniquely rewritten as:
	\[
	    \bf{d} = [\bf{T}_1, \bf{T}_2, \ldots, \bf{T}_m]
	\]
	where the blocks $\bf{T}_i$ correspond to the types listed above, for type B we further require only $\bf{T}_m$ is of Type B3.
\end{lemma}

\begin{remark}
	These decompositions play a very crucial role in the parabolic Hitchin system, where they provide a framework for understanding the structure and decompositions of local parabolic Higgs bundles. For more details, see~\cite{WWW24,WWW25}.
\end{remark}

By Springer\cite[Theorem 6.10]{Spr}, there is an injective map from the irreducible representations of the Weyl group to equivalence classes of pairs of nilpotent elements and irreducible representations of the component group. Specifically:  
\[
    \mathrm{Irr}(W) \hookrightarrow \{ (e, \rho) \}/G,
\]
where $e \in \g$ is a nilpotent element, and $\rho \in \mathrm{Irr}(A(\bf{O}_e))$. Here, $A(\bf{O}_e) = G^e_{ad} / (G^e_{ad})^\circ $ is the component group associated with the adjoint group $G_{ad}$. For pairs $(X, \rho)$ coming from the special representation, $\rho$ is automatically trivial. In such cases, we call the nilpotent orbit $\bf{O}_e$ \emph{special}. 

Given a partition $\bf{d}=\left[ d_1, \ldots, d_r  \right]$, its transpose partition $\bf{d}^t$ is defined as 
\[
    d^t_i = \sharp \left\{ j \ \big{|} \ d_j \geq i \right\}, \quad \text{for all} \; i.
\]     
\begin{definition}\label{def:special nilpotent}
    We call a partition $\bf{d} \in \mathcal{P}_{\varepsilon}(N)$ {\emph{special}} if 
    \begin{itemize}
        \item For $\varepsilon = -1$ and $N=2n$, the transpose partition $\bf{d}^t$ lies in $\mathcal{P}_{-1}(2n)$,

        \item For $\varepsilon = 1$ and $N=2n+1$, the transpose partition $\bf{d}^t$ lies in $\mathcal{P}_{1}(2n+1)$,

        \item For $\varepsilon = 1$ and $N=2n$, the transpose partition $\bf{d}^t$ lies in $\mathcal{P}_{-1}(2n)$.
    \end{itemize}
\end{definition}

Let $\bf{d} = [\bf{T}_1, \bf{T}_2, \ldots, \bf{T}_m]$ be as in Lemma~\ref{lem.building block}. Then, we have the following characterization of special partitions:

\begin{lemma}\
	\begin{itemize}
		\item Type B: $\bf{d}$ is special if and only if it contains no block of Type B1*.
		\item Type D: $\bf{d}$ is special if and only if it contains no block of Type D2 of the form $[\alpha_1, \beta_1, \beta_1, \ldots, \beta_k, \beta_k, \alpha_2]$ with $k \geq 1$.
		\item Type C: $\bf{d}$ is special if and only if it contains no block of Type C2 or the form $[\beta_1, \alpha_1, \alpha_1, \ldots, \alpha_k, \alpha_k, \beta_2]$ with $k \geq 1$.
	\end{itemize}
\end{lemma}
\begin{proof}
	This follows directly from an alternative characterization of special partitions given in \cite[p.100]{CM93}. In particular, the stated conditions on the absence of blocks (Type B1*, Type D2 with $k \geq 1$, and Type C2 with $k \geq 1$) agree with the characterization in terms of the transpose partition.
\end{proof}

Among the special orbits, some orbits are even more special, namely, \emph{Richardson orbits}. Let $G$ be $\SO_{2n+1}$, $\Sp_{2n}$, or $\SO_{2n}$. We say a nilpotent orbit $\bf{O}_{R} \subset \g$ \emph{Richardson} if there exists a parabolic subgroup $P < G$ such that the image of the moment map (also called \emph{generalized Springer map})
\begin{align}\label{Springer map}
    T^*(G/P) \longrightarrow \overline{\bf{O}}_R,
\end{align}
is the closure of $\bf{O}_R$. The parabolic subgroup $P$ is called a \emph{polarization} of the Richardson orbit $\bf{O}_R$. 

Let $\langle \cdot, \cdot \rangle_{\varepsilon} $, $\varepsilon=1$ (resp. $-1$), be the bilinear form on $\mathbb{C}^N$ for $N=2n+1$ or $2n$ (resp. $N=2n$), depending on the type of Lie algebra $\g$. Every parabolic subgroup of $G$ is the stabilizer of some \emph{admissible isotropic flag}:\footnote{The \emph{admissible condition} requires that when $\varepsilon=1$ then $q \neq 2$.}
\begin{align}\label{isotropic flag}
	\mathbb{C}^N=F_0 \supset F_1 \dots \supset F_{k-1} \supset  F_k \supset F_k^{\perp} \supset F_{k-1}^{\perp} \supset \dots \supset F_1^{\perp} \supset F_0^{\perp} = \{0\},
\end{align}
where $F_i^{\perp}$ denotes the orthogonal complement of $F_i$ with respect to the bilinear form $ \langle \cdot, \cdot \rangle_{\varepsilon} $. Let
\begin{align*}
	p_i = \dim (F_i/F_{i-1}) \quad \text{for} \ 1 \leq i\leq k, \quad \text{and} \quad q = \dim (F_{k}^{\perp}/F_k).
\end{align*}
Note that $2\sum_{i=1}^k p_i +q = N $, and
\begin{align}\label{levi-type}
 \mathfrak{l} \cong \mathfrak{gl}_{p_1} \times \ldots \times \mathfrak{gl}_{p_k} \times \g^{\prime},	
\end{align}
where $\g^{\prime}$ has the same type as $\g$. We further require $p_1 \leq \ldots \leq p_k$, and we say the flag is of type $(p_1, \ldots, p_k; q)$. 

If $P<G$, with Lie algebra $\gp$, is a polarization of $\bf{O}_R$, the partition of $\bf{O}_R$ is induced by the partition ${\vec { \bf{1} } }= \{ \bf{1}_1, \ldots \bf{1}_k ; \bf{1}_0 \} $, where each $\bf{1}_i= \left[ 1_{i1}, \ldots, 1_{i p_i} \right]$ corresponds to a partition of $p_i$ (consisting of ones) in $\mathfrak{gl}_{p_i}$ if $i=1, \ldots, k$, and $\bf{1}_0= \left[ 1_{i1}, \ldots, 1_{i r'_i} \right]$ is a partition in $\mathfrak{g}^\prime$. Consider the new partition
\begin{align}
    \bf{d} &= \left[ d_1, \ldots, d_r \right],	\quad \text{with} \quad d_j = 2 \sum_{i=1}^k 1_{ij} + 1_{0j}. \label{induction}
\end{align}
Here, we use the convention that if the subscript $j$ exceeds the range of the subscripts of $\bf{1}_i$, we take $1_{ij}=0$. The parition of $\bf{O}_R$ is the B/C/D-collapse of $\bf{d}$, denoted by $\bf{d}_{B/C/D}$. See \cite[\S~7]{CM93} for more details.

It is known that under conjugation, there are finite many polarizations. Two parabolic subgroups are called \emph{associated} if they have conjugate Levi subgroups. Denote by $\text{Pol}(\bf{O}_R)$ the set of representitives of its associated polarizations. This set plays an important role in the mirror symmetry of parabolic covers; see \cite{FRW24} for more details.

The partition $\bf{d}$ has the following property: it can be divided into two parts based on parity. Specifically:
\begin{align}\label{structure of induction}
\bf{d} = [\mathbf{Odd} \mid \mathbf{Even}] 
\end{align}
where
\begin{align*} 
\mathbf{Odd}= [ d_{1} \equiv \ldots \equiv d_{l} \equiv 1 ] \quad \text{and} \quad   \mathbf{Even}=[d_{l+1} \equiv \ldots \equiv d_{r} \equiv 0] .	
\end{align*} 
This observation can be used to determine which nilpotent orbits are Richardson orbits. Spaltenstein was the first to compute polarizations of Richardson orbits (unpublished). Hesselink provided a proof and description in \cite{He78}, and Kempken offered another explanation in \cite{Ke83}. The description of Richardson orbits of types B, C, and D is as follows:
\vspace{0.3 cm}

\textbf{Type B}:

A partition $\bf{d}_B = [\bf{T}_1, \bf{T}_2, \ldots, \bf{T}_m]$ is Richardson if 
\begin{itemize}
	\item[1.] There is no block of type B1*.
	\item[2.] There exists an integer $l$ such that:
	\begin{itemize}
	    \item[$\bullet$] For $i < l$, $\bf{T}_i$ is either type B1 or type B2 of the form $[\alpha_1, \alpha_2]$. 
	    \item[$\bullet$] For $i=l$, $\bf{T}_l$ is type B2 of the form $[\alpha_1, \beta_1, \beta_1, \ldots, \beta_k, \beta_k, \alpha_2]$ with $k \geq 1$.
	    \item[$\bullet$] For $l+1 \leq i \leq m-1$, $\bf{T}_i$ is type B2. 
	    \item[$\bullet$] For $i = m$, $\bf{T}_m$ is type B3.
	\end{itemize} 
\end{itemize}

\textbf{Type C}:

A partition $\bf{d}_C = [\bf{T}_1, \bf{T}_2, \ldots, \bf{T}_m]$ is Richardson if
\begin{itemize}
	\item[1.] There is no block of type C2 of the form $[\beta_1, \alpha_1, \alpha_1, \ldots, \alpha_k, \alpha_k, \beta_2]$ with $k \geq 1$.
	\item[2.] There exists an integer $l$ such that:
	\begin{itemize}
	    \item[$\bullet$] For $i < l$, $\bf{T}_i$ is either type C1, Type C1* or type C2 of form $[\beta_1, \beta_2]$ with the additional condition that the last part of $\bf{T}_i$ and the first part of $\bf{T}_{i+1}$ are different. 
	    \item[$\bullet$] For $i =l$, $\bf{T}_l$ is type C1. 
	    \item[$\bullet$] For $i \geq l+1$, $\bf{T}_i$ is type C1*.
	\end{itemize}
\end{itemize}

\textbf{Type D}:

A partition $\bf{d}_D = [\bf{T}_1, \bf{T}_2, \ldots, \bf{T}_m]$ is Richardson if
\begin{itemize}
	\item[1.] There is no block of type D2 of the form $[\alpha_1, \beta_1, \beta_1, \ldots, \beta_k, \beta_k, \alpha_2]$ with $k \geq 1$.
	\item[2.] There exists an integer $l$ such that:
	\begin{itemize}
	    \item[$\bullet$] For $i < l$, $\bf{T}_i$ is either type D1 or type D2 of form $[\alpha_1, \alpha_2]$.
	    \item[$\bullet$] For $i = l$, $\bf{T}_l$ is type D1*.
	    \item[$\bullet$] For $i \geq l+1$, $\bf{T}_i$ is either type D1, type D1* or type D2 of form $[\alpha_1, \alpha_2]$ with the additional condition that the last part of $\bf{T}_i$ and the first part of $\bf{T}_{i+1}$ are different. 
	\end{itemize}
\end{itemize}

\section{Minimal Richardson and pseudo-polarizations} \label{sec.mini_R}

In this subsection, we generalize the concept of polarizations for Richardson orbits to encompass all nilpotent orbits.

\begin{definition}[Minimal Richardson orbit]\label{def.mimi_Richardson}
	 Let $\mathbf{O}_{e}$ be a nilpotent orbit of type B, C, or D. A \emph{minimal Richardson orbit} $\mathbf{O}_R$ for $\mathbf{O}_{e}$  is a Richardson orbit of the same type such that $\overline{\mathbf{O}}_R \supset \mathbf{O}_{e} $ and $\mathbf{O}_R$ is minimal with respect to the partial order~\eqref{partial-order}.
\end{definition}

\begin{definition}[pseudo-polarization]\label{def.pseudo-polarization}
    A \emph{pseudo-polarization} of $\mathbf{O}_{e}$ is any polarization $P$ of any minimal Richardson orbit $\mathbf{O}_R$ for $\mathbf{O}_{e}$.
\end{definition}

In the following, given a nilpotent orbit $\bf{O}_e$, we provide a characterization of all its minimal Richardson orbits. The characterization of pseudo-polarizations will be addressed.
\vspace{0.2 cm}

\textbf{Type B}: Let $\bf{d}_B = [d_1, d_2, \ldots, d_r] = [\bf{T}_1, \bf{T}_2, \ldots, \bf{T}_m]$ be the partition associated with $\bf{O}^B_e$, a nilpotent orbit of type B. Below, we introduce some modifications of the blocks $\bf{T}_i$ which are useful in constructing minimal Richardson orbits. 

Modifications of blocks for type B:
\begin{itemize}
	\item[1.] For blocks of type B1:\\ 
	If $\bf{T}_i$ is of type B1, i.e., $\bf{T}_i = [\alpha, \alpha]$, we define
    \[
         \bf{T}'_i = [\alpha+1, \alpha-1], \quad  \bf{T}''_i = [\alpha, \alpha],
    \]
    
    \item[2.] For blocks of type B1*:\\ 
	If $\bf{T}_i$ is of Type B1*, i.e., $\bf{T}_i = [\beta, \beta]$, we define
	\[
	    \bf{T}^{\circ}_i = [\beta+1, \beta],
	\]
	and
    \[
         \bf{T}'_i = [\beta, \beta], \quad  \bf{T}''_i = [\beta+1, \beta-1],
    \]

    \item[3.] For blocks of type B2:\\
    If $\bf{T}_i$ is of type B2, i.e., $\bf{T}_i = [\alpha_1, \beta_1, \beta_1, \ldots, \beta_k, \beta_k, \alpha_2]$, we define
    \[
        \bf{T}^\circ_i = [\alpha_1, \beta_1, \beta_1, \ldots, \beta_k, \beta_k, \alpha_2+1],
    \]
    \[
        \bf{T}'_i   =  [\alpha_1-1, \beta_1, \beta_1, \ldots, \beta_k, \beta_k, \alpha_2+1],
    \]
    and 
    \[
        \bf{T}''_i=[\alpha_1, \beta'_1, \beta'_1, \ldots, \beta'_k, \beta'_k, \alpha_2]
    \]
    where $\bf{T}''_i$ is the rearrangement of
    \[
         [\alpha_1, \beta_1+1, \beta_1-1, \ldots, \beta_k+1, \beta_k-1, \alpha_2]
    \]
    such that $\beta'_1 \geq \beta'_2 \geq \ldots \geq \beta'_k$.
    
    \item[4.] For the block of type B3:\\
    The only block of Type B3 is $\bf{T}_m = [\alpha, \beta_1, \beta_1, \ldots, \beta_k, \beta_k]$. We define
    \[
        \bf{T}^\circ_i = [\alpha, \beta_1, \beta_1, \ldots, \beta_k, \beta_k],
    \]
    and
    \[
        \bf{T}'_i = [\alpha-1, \beta_1, \beta_1, \ldots, \beta_k, \beta_k].
    \]
\end{itemize}

Let 
\[
    I^B_e = \{ \bf{T}_{i_1}, \bf{T}_{i_2}, \ldots \} 
\] 
be a subset of $\{ \bf{T}_1, \bf{T}_2, \ldots, \bf{T}_m \}$, defined inductively as follows. In this process, we keep $\bf{T}_m$ to be $[\alpha, \beta_1, \ldots, \beta_k]$ even if $\beta_1 = \ldots = \beta_k = 0$.
\begin{itemize}
	\item Initial block $\bf{T}_{i_1}$:\\
	Let $\bf{T}_{i_1}$ be the first block containing $d_{2l_1}$ such that 
	\[
	    d_{2l_1} = d_{2l_1+1} \equiv 0 \quad \text{for some}\; l_1.
	\]
	\item Subsequent blocks $\bf{T}_{i_j}$:\\
	For $j>1$, let $\bf{T}_{i_j}$ be the $j$-th block containing $d_{2l_j}$ such that 
	\[
	    d_{2l_j} = d_{2l_j+1} \equiv 0 \quad \text{for some}\; l_j,
	\]
	and there exists $l'_j$ satisfying 
	\[
	    l_{j-1} < l'_j < l_j, \quad \text{with}\;  d_{2l'_j+1}=d_{2l'_j+2} \equiv 1.
	\]
\end{itemize}

For $\bf{T}_{i_a} \in I^B_e$, define the partition 
\[
    \bf{d}_{i_a}= [d'_1, d'_2, \dots, d'_r]
\]
as the rearragement of 
\[
    [ \bf{T}''_1, \ldots, \bf{T}''_{i_a-1}, \bf{T}^\circ_{i_a}, \bf{T}'_{i_a+1}, \ldots, \bf{T}'_m]
\]
such that $d'_1 \geq d'_2 \geq \ldots \geq d'_r$. 

The minimal Richardson orbit associated with $\bf{T}_{i_a}$ is the B-collapse of $\bf{d}_{i_a}$, denoted by 
\[
    \bf{d}^B_{i_a} := (\bf{d}_{i_a})_B.
\]

\begin{proposition}\label{prop.mim_R_B}
	All minimal Richardson orbits of $\bf{O}^B_e$ are characterized by $I^B_e$, i.e., they are given by the collection
	\[
	    \{ \bf{O}_{\bf{d}^B_{i_a}} \}_{\bf{T}_{i_a} \in I^B_e}.
	\]
\end{proposition}
\begin{proof}
	First, we show that the partitions $\bf{d}^B_{i_a}$ are pairwise incomparable under the partial order~\eqref{partial-order}. Let $i_\alpha < i_\beta$ be two indices in $I^B_e$. From the definition of $I^B_e$ and the block types B1, B1*, B2, and B3, we observe that $i_\beta \geq i_\alpha +2$. This implies that the blocks $\bf{T}_{i_\alpha}$ and $\bf{T}_{i_\beta}$ are separated by at least one block. The partition $\bf{d}$ of $\bf{O}^B_e$ can be expressed as
	\[
		\bf{d}_e = [ \ldots, d_{2l_\alpha} = d_{2l_\alpha+1} \equiv 0, \ldots, d_{2l+1}=d_{2l+2} \equiv 1, \ldots, d_{2l_\beta} = d_{2l_\beta+1} \equiv 0, \ldots ],
	\]
	where $d_{2l_\alpha} \in \bf{T_{i_\alpha}}$ and $d_{2l_\beta} \in \bf{T_{i_\beta}}$. Then
	\begin{align*}
	    \bf{d}_{i_\alpha} &= [ \ldots, d_{2l_\alpha}, d_{2l_\alpha+1}, \ldots, d_{2l+1}+1, d_{2l+2}-1, \ldots, d_{2l_\beta}, d_{2l_\beta+1}, \ldots ], \\
	    \bf{d}_{i_\beta} &= [ \ldots, d_{2l_\alpha}+1, d_{2l_\alpha+1}-1, \ldots, d_{2l+1}, d_{2l+2}, \ldots, d_{2l_\beta},  d_{2l_\beta+1}, \ldots ].
	\end{align*}
	Since $[d_{2l_\alpha}+1, d_{2l_\alpha+1}-1] \geq [d_{2l_\alpha}, d_{2l_\alpha+1}]$ and $[d_{2l+1}, d_{2l+2}] \leq [d_{2l+1}+1, d_{2l+2}-1]$, $\bf{d}_{i_\alpha}$ and $\bf{d}_{i_\beta}$ have no relation under the partial order~\eqref{partial-order}. But one needs to be careful that the B-collapse may change $d_{2l+1}+1$ or $d_{2l+2}-1$. Due to the observation: $i_\beta \geq i_\alpha +2$, there are two cases:
	\begin{enumerate}
		\item $\bf{d}_e = [\ldots, d_{2l} \equiv 1, d_{2l+1} = d_{2l+2} \equiv 1, \ldots ]$, i.e., there is a type B1 block before $d_{2l+1}$.
		\item $\bf{d}_e = [\ldots, d_{2l-1} = d_{2l} \equiv 0, d_{2l+1} = d_{2l+2} \equiv 1, \ldots ]$, i.e., there is a type B1* block before $d_{2l+1}$.
	\end{enumerate}
	In the first case, $\bf{d}_{i_\alpha}$ and $\bf{d}_{i_\beta}$ become
	\begin{align*}
		\bf{d}_{i_\alpha} &= [\ldots, d_{2l} + 1, d_{2l+1}+1,  d_{2l+2} - 1, \ldots ],\\
		\bf{d}_{i_\beta}  &= [\ldots, d_{2l} , d_{2l+1},  d_{2l+2}, \ldots ].
	\end{align*}
	The B-collapse of $\bf{d}_{i_\alpha}$ happens at $d_{2l}+1$, so that $\bf{d}_{i_\alpha}^B=[\ldots, d'_{2l}, d'_{2l+1}, \ldots] = [\ldots, d_{2l}, d_{2l+1}+1, \ldots]$. Thus, $\bf{d}_{i_\alpha}^B$ and $\bf{d}_{i_\beta}^B$ are incomparable under the partial order~\eqref{partial-order}. Similarly, the same conclusion holds in the second case. Therefore, $\bf{d}_{i_\alpha}^B$ and $\bf{d}_{i_\beta}^B$ are incomparable under the partial order~\eqref{partial-order}.

	Next, we show that, for any partition $\bf{d}_R$ of Richardson orbit $\bf{O}_R$ such that $\bf{d}_R \geq \bf{d}_e$, there exists $\bf{T}_{i_a} \in I^B_e$ such that $\bf{d}_R \geq \bf{d}^B_{i_a}$. Since $\bf{d}_R$ is Richardson, it has no block of type B1* and there exists an integer $s$
	\[
		\bf{d}_R = [\bar{d}_1, \ldots,\bar{d}_{2s-1} \ldots, \bar{d}_{r'}] = [\bar{\bf{T}}_1, \ldots, \bar{\bf{T}}_{h-1}, \bar{\bf{T}}_h, \bar{\bf{T}}_{h+1}, \ldots, \bar{\bf{T}}_{m'}]
	\] 
	satisfying the conditions of Richardson orbits of type B, see \S~\ref{Sec:preliminaries}. Here we suppose $\bar{d}_{2s-1} \equiv 1$ be the first part in $\bar{\bf{T}}_h$. Then,
	\[
		\bf{d}_P = [\bar{\bf{T}}''_1, \ldots, \bar{\bf{T}}''_{h-1}, \bar{\bf{T}}^\circ_h, \bar{\bf{T}}'_{h+1}, \ldots, \bar{\bf{T}}'_{m'}]
	\]
	is an induced partition from a polarization $P$ of $\bf{O}_R$. Let $l_\alpha$ be the smallest integer larger than $s$ so that $0 \equiv d_{2l_\alpha} \in \bf{T}_{i_\alpha}$ for some $\bf{T}_{i_\alpha}$ in $I_e^B$, i.e.,
	\[
		\bf{d}_e = [d_1, \ldots, d_{2l_\alpha} = d_{2l_\alpha+1}, \ldots, d_r] = [\bf{T}_1, \ldots, \bf{T}_{i_\alpha}, \ldots, \bf{T}_m].
	\]
	Then, it is sufficient to show
	\[
		\bf{d}_P := [\bar{d}'_1, \ldots, \bar{d}'_{r'}] \geq \bf{d}_{i_\alpha} := [d'_1, \ldots, d'_r] = [ \bf{T}''_1, \ldots, \bf{T}''_{i_\alpha-1}, \bf{T}^\circ_{i_\alpha}, \bf{T}'_{i_\alpha+1}, \ldots, \bf{T}'_m].
	\] 

	Since $\bf{d}_R \geq \bf{d}_e$, it is easy to show $\sum_{i=1}^j \bar{d}'_i \geq \sum_{i=1}^j d'_i$ for all $j \leq 2s-2$. If 
	\begin{enumerate}
		\item $d_{2s-1} \equiv 1$ and $d_{2s} = d_{2s+1} \equiv 0$;
		\item $d_{2s-1} \equiv d_{2s} \equiv 1$;
		\item $d_{2s-1} = d_{2s} \equiv 0$, and $d_{2s} = d_{2s+1}$;
		\item $d_{2s-1} = d_{2s} \equiv 0$, and $d_{2s} \neq d_{2s+1}$.
	\end{enumerate}
	In the both cases (1) ($l_\alpha = s$) and (2), $d'_{2s-1} = d_{2s-1}$, $d'_{2s} = d_{2s}$, then $\sum_{i=1}^j \bar{d}'_i \geq \sum_{i=1}^j d'_i$ for all $j \leq 2s$ follows from $\bf{d}_R \geq \bf{d}_e$. In case (3), $l_\alpha = s+1$, $d'_{2s-1} = d_{2s-1}+1$, $d'_{2s} = d'_{2s+1} = d'_{2s+2} = d_{2s}$, then $\sum_{i=1}^j \bar{d}'_i \geq \sum_{i=1}^j d'_i$ for all $j \leq 2s+2$. In case (4), $d'_{2s-1} = d_{2s-1}+1$, $d'_{2s} = d_{2s}-1$, then $\bar{d}'_{2s-1} \geq d'_{2s-1}$ and $\bar{d}'_{2s} \geq d'_{2s}+1$. If $d_{2s+1}$ is in the cases (1), (2) or (4), it is easy to show $\sum_{i=1}^j \bar{d}'_i \geq \sum_{i=1}^j d'_i$ for all $j \leq 2s+2$. If $d_{2s+1}$ is in the case (3), then $l_\alpha = s+1$, $d'_{2s+1} = d_{2s+1}+1$, $d'_{2s+2} = d'_{2s+3} = d'_{2s+4} = d_{2s+1}$, and $\sum_{i=1}^j \bar{d}'_i \geq \sum_{i=1}^j d'_i$ for all $j \leq 2s+4$ since $\bar{d}'_{2s} \geq d'_{2s}+1$. Continuing this process, we finally reach $l_\alpha$ and show $\sum_{i=1}^j \bar{d}'_i \geq \sum_{i=1}^j d'_i$ for all $j \leq 2l_\alpha$. After that, due to the construction of $\bf{T}'_{j}$, it can be shown that $\sum_{i=1}^j \bar{d}'_i \geq \sum_{i=1}^j d'_i$ for all $j$. Thus, $\bf{d}_P \geq \bf{d}_{i_\alpha}$. 
\end{proof}

\textbf{Type C}: Let $\bf{d}_C = [d_1, d_2, \ldots, d_r] = [\bf{T}_1, \bf{T}_2, \ldots, \bf{T}_m]$ be the partition associated with $\bf{O}^C_e$, a nilpotent orbit of type C. Similarly, we have the following modifications of the blocks $\bf{T}_i$:

\begin{itemize}
	\item[1.] For blocks of type C1:\\
	if $\bf{T}_i = [\alpha, \alpha]$ is of Type C1, we define
	\[
	    \bf{T}'_i = [\alpha+1, \alpha-1], \quad \bf{T}''_i = [\alpha, \alpha].
	\] 
	
	\item[2.] For blocks of type C1*:\\
	if $\bf{T}_i = [\beta, \beta]$ is of Type C1*, we define
	\[
	    \bf{T}^\circ_i=\bf{T}'_i = [\beta, \beta],\quad \bf{T}''_i = [\beta+1, \beta-1].
	\]
	
	\item[3.] For blocks of type C2:\\
	if $\bf{T}_i = [\beta_1, \alpha_1, \alpha_1, \ldots, \alpha_k, \alpha_k, \beta_2]$ is of Type C2, we define
	\[
	    \bf{T}'_i = [\beta_1, \alpha'_1, \alpha'_1, \ldots, \alpha'_k, \alpha'_k, \beta_2]
	\]
	where $\bf{T}'_i$ is the rearrangement of
	\[
	    [\beta_1, \alpha_1+1, \alpha_1-1, \ldots, \alpha_k+1, \alpha_k-1, \beta_2]
	\]
	such that $\alpha'_1 > \alpha'_2> \ldots > \alpha'_k$. Additionally, we define
    \[
        \bf{T}^\circ_i=\bf{T}''_i=[\beta_1+1, \alpha_1, \alpha_1, \ldots, \alpha_k, \alpha_k, \beta_2-1].
    \]
\end{itemize}

Let
\[
    I^C_e = \{ \bf{T}_{i_1}, \bf{T}_{i_2}, \ldots \} 
\] 
be a subset of $\{ \bf{T}_1, \bf{T}_2, \ldots, \bf{T}_m \}$, defined inductively:
\begin{itemize}
	\item Initial block $\bf{T}_{i_1}$:\\ 
	Let $\bf{T}_{i_1}$ be the first block containing $d_{2l_1}$ satisfying 
	\[
	    d_{2l_1} = d_{2l_1+1} \equiv 0 \quad  \text{for some}\; l_1.
	\]
	\item Subsequent blocks $\bf{T}_{i_j}$:\\
	For $j>1$, let $\bf{T}_{i_j}$ be the $j$-th block containing $d_{2l_j}$ satisfying 
	\[
	    d_{2l_j} = d_{2l_j+1} \equiv 0 \quad  \text{for some}\; l_j.
	\]
	and there exists $l'_j$ satisfying 
	\[
	    l_{j-1} < l'_j < l_j,  \quad \text{with} \; d_{2l'_j+1}=d_{2l'_j+2} \equiv 1.
	\]
\end{itemize}

For $\bf{T}_{i_a} \in I^C_e$, define the partition 
\[
    \bf{d}_{i_a}= [d'_1, d'_2, \ldots, d'_r]
\]
as the rearragement of 
\[
    [ \bf{T}''_1, \ldots, \bf{T}''_{i_a-1}, \bf{T}^\circ_{i_a}, \bf{T}'_{i_a+1}, \ldots, \bf{T}'_m]
\]
such that $d'_1 \geq d'_2 \geq \ldots \geq d'_r$.

The minimal Richardson orbit associated with $\bf{T}_{i_a}$ is the C-collapse of $\bf{d}_{i_a}$, denoted by 
\[
    \bf{d}^C_{i_a} := (\bf{d}_{i_a})_C.
\]

\begin{proposition}\label{prop.mim_R_C}
	All minimal Richardson orbits of $\bf{O}^C_e$ are characterized by $I^C_e$, i.e., they are given by the collection
	\[
	    \{ \bf{O}_{\bf{d}^C_{i_a}} \}_{\bf{T}_{i_a} \in I^C_e}.
	\]
\end{proposition}
\begin{proof}
	The proof is similar to that of Proposition~\ref{prop.mim_R_B}.
\end{proof}

\textbf{Type D}: Let $\bf{d}_D = [d_1, d_2, \ldots, d_r] = [\bf{T}_1, \bf{T}_2, \ldots, \bf{T}_m]$ be the partition associated with $\bf{O}^D_e$, a nilpotent orbit of type D. Similarly, we have the following modifications of the blocks $\bf{T}_i$:

\begin{itemize}
	\item[1.] For blocks of type D1:\\
	if $\bf{T}_i = [\alpha, \alpha]$ is of type D1, we define
	\[
	    \bf{T}'_i = [\alpha+1, \alpha-1], \quad \bf{T}''_i = [\alpha, \alpha].
	\] 
	
	\item[2.] For blocks of type D1*:\\
	if $\bf{T}_i = [\beta, \beta]$ is of type D1*, we define
	\[
	    \bf{T}^\circ_i =\bf{T}'_i = [\beta, \beta],\quad \bf{T}''_i = [\beta+1, \beta-1].
	\]
	
	\item[3.] For blocks of type D2:\\
	if $\bf{T}_i = [\alpha_1, \beta_1, \beta_1, \ldots, \beta_k, \beta_k, \alpha_2]$ is of type D2, we define
	\[
	    \bf{T}^\circ_i = [\alpha_1, \beta_1+1, \beta_1, \beta_2, \beta_2 \ldots, \beta_k, \beta_k, \alpha_2-1],
	\]
	and
	\begin{align*}
	    \bf{T}'_i  &= [\alpha_1+1, \beta_1, \beta_1, \ldots, \beta_k, \beta_k, \alpha_2-1], \\
        \bf{T}''_i &=[\alpha_1, \beta'_1, \beta'_1, \ldots, \beta'_k, \beta'_k, \alpha_2] 	
    \end{align*}
    where $\bf{T}''_i$ is the rearrangement of
    \[
         [\alpha_1, \beta_1+1, \beta_1-1, \ldots, \beta_k+1, \beta_k-1, \alpha_2]
    \]
    such that $\beta'_1 \geq \beta'_2 \geq \ldots \geq \beta'_k$.
\end{itemize}

Let
\[
    I^D_e = \{ \bf{T}_{i_1}, \bf{T}_{i_2}, \ldots \} 
\] 
be a subset of $\{ \bf{T}_1, \bf{T}_2, \ldots, \bf{T}_m \}$, defined inductively:
\begin{itemize}
	\item Initial block $\bf{T}_{i_1}$:\\ 
	Let $\bf{T}_{i_1}$ be the first block containing $d_{2l_1-1} = d_{2l_1} \equiv 0$ for some $l_1$.
	\item Subsequent blocks $\bf{T}_{i_j}$:\\
	For $j>1$, let $\bf{T}_{i_j}$ be the $j$-th block containing $d_{2l_j-1} = d_{2l_j} \equiv 0$ for some $l_j$, such that there exists $l'_j$ satisfying 
	\[
	    l_{j-1} < l'_j < l_j,  \quad \text{with} \; d_{2l'_j}=d_{2l'_j+1} \equiv 1.
	\]
\end{itemize}

For $\bf{T}_{i_a} \in I^D_e$, define the partition 
\[
    \bf{d}_{i_a}= [ d'_1, d'_2, \dots, d'_r ]
\]
as the rearragement of
\[
    [\bf{T}''_1, \ldots, \bf{T}''_{i_a-1}, \bf{T}^\circ_{i_a}, \bf{T}'_{i_a+1}, \ldots, \bf{T}'_m]
\]
such that $d'_1 \geq d'_2 \geq \ldots \geq d'_r$. 

The minimal Richardson orbit associated with $\bf{T}_{i_a}$ is the D-collapse of $\bf{d}_{i_a}$, denoted by 
\[
    \bf{d}^D_{i_a} := (\bf{d}_{i_a})_D.
\]

\begin{proposition}\label{prop.mim_R_D}
	All minimal Richardson orbits of $\bf{O}^D_e$ are characterized by $I^D_e$, i.e., they are given by the collection
	\[
	    \{ \bf{O}_{\bf{d}^D_{i_a}} \}_{\bf{T}_{i_a} \in I^D_e}.
	\]
\end{proposition}
\begin{proof}
	The proof follows similarly to that of Proposition~\ref{prop.mim_R_B}.
\end{proof}

\begin{remark}\label{rem.pseudo-polarization}
    The set of associated polarizations can be characterized by the degrees of the generalized Springer maps \eqref{Springer map}, a concept referred to as the \emph{footprint}. For more details on the notion of \emph{footprint}, see \cite[\S~5.1]{FRW24}. Combining \cite[Lemma 5.3]{FRW24} with the results above, one can explicitly determine the footprint.
\end{remark}

\section{The geometry of Spaltenstein varieties} \label{s.Spal_fiber}

Let $\bf{O}_e$ be a nilpotent orbit of type B, C, or D, and $\bf{O}_R$ be a minimal Richardson orbit associated with an element in $I^{\bullet}_e$. Taking an associated polarization $P$:
\[
    \mu_P: T^*(G/P) \rightarrow \overline{\bf{O}}_R \supset \bf{O}_e.
\]
In this subsection, we study the Spaltenstein variety $\mu_P^{-1}(e)$.

Before describing the geometry of Spaltenstein variety, we need the following useful lemmas:

\begin{lemma}\label{lem:isotropic subspace}
	Let $V_i=\CC^{n_i}$, for $i=1,\ldots,k$, endow with nondegenerate symmetric pairings. Consider a set $\mathbf{W}$ of subspaces $W \subset \bigoplus_{i=1}^k V_i$ such that
	\begin{itemize}
		\item $W$ is isotropic,
		\item $\dim \left( W \bigcap \bigoplus_{i=1}^j V_i \right) = \big\lfloor \frac{\sum_{i=1}^j n_i}{2} \big\rfloor$,\footnote{Here $\lfloor \frac{n}{2} \rfloor$ means the greatest integer that smaller than $\frac{n}{2}$.} for $j=1,\ldots,k$.
	\end{itemize}
	Then, $\mathbf{W} \subset \mathrm{OG}\big(\big\lfloor \tfrac{\sum_{i=1}^k n_i}{2} \big\rfloor, \bigoplus_{i=1}^k V_i \big)$. Let $\mathbf{W}_j=\mathbf{W}\cap \left(\oplus_{i=1}^j V_i \right)$, then 
	\[
		\mathbf{W}=\mathbf{W}_k \rightarrow \mathbf{W}_{k-1} \rightarrow \cdots \rightarrow \mathbf{W}_1,
	\]
	where $\mathbf{W}_j \rightarrow \mathbf{W}_{j-1}$ is a $\mathrm{OG}(m_j, N_j )$ fibration, here $m_j=\lfloor \frac{\sum_{i=1}^j n_i}{2} \rfloor - \lfloor \frac{\sum_{i=1}^{j-1} n_i}{2} \rfloor$ and $N_j= \sum_{i=1}^j n_i - 2\lfloor \frac{\sum_{i=1}^{j-1} n_i}{2} \rfloor$.
\end{lemma}
\begin{proof}
	Fix an integer $1\leq l<k$ and set $\tilde{V}_1 = \bigoplus_{i=1}^{l} V_i$ and $\tilde{V}_2=\bigoplus_{i=l+1}^k V_i$. First notice that if $\dim \tilde{V}_1 \equiv \dim \tilde{V}_2 \equiv 0$, then $W=(W\cap \tilde{V}_1) \oplus (W\cap \tilde{V}_2)$. Indeed, since $W$ is isotropic, we have $W \subset (W \cap \tilde{V}_1)^\perp \oplus \tilde{V}_2$, where $(W \cap \tilde{V}_1)^\perp$ is the orthogonal complement of $W \cap \tilde{V}_1$. Note that $\dim (W \cap \tilde{V}_1) = \frac{1}{2} \dim \tilde{V}_1$, so $W \cap \tilde{V}_1$ is maximal isotropic. Hence, $(W \cap \tilde{V}_1)^\perp = W \cap \tilde{V}_1$.

	Then, the general case can be proved by induction from the following simple case: $W \subset V_1 \oplus V_2$ where $\dim V_1 \equiv 1$ or $\dim V_2 \equiv 1$. Without loss of generality, suppose $\dim V_1 \equiv 1$, $\mathbf{W}_1 = \mathrm{OG}(\lfloor \frac{n_1}{2} \rfloor, V_1)$. We have a map $\phi: \mathbf{W} \rightarrow \mathbf{W}_1$ sending $W \in \mathbf{W}$ to $W_1:= (W \cap V_1) \in \mathbf{W}_1$. As mentioned above, $W \subset (W \cap V_1)^\perp \oplus V_2$, then it can be shown that $\phi^{-1}(W_1) \cong \mathrm{OG}(\lfloor \frac{n_1+n_2}{2} \rfloor - \lfloor \frac{n_1}{2} \rfloor, W_1^\perp /W_1 \oplus V_2)$.
\end{proof}

\begin{lemma}\label{lem:Lagrangian subspace}
	Let $V'_i=\CC^{n'_i}$, for $i=1,\ldots,k'$, endow with nondegenerate skew symmetric pairings. Then $n'_i \equiv 0$, let $m'_i= \frac{n'_i}{2}$. Consider a set $\mathbf{W}'$ of subspaces $W' \subset \bigoplus_{i=1}^{k'} V'_i$ such that
	\begin{itemize}
		\item $W'$ is isotropic,
		\item $\dim \left( W' \bigcap \bigoplus_{i=1}^j V'_i \right) = \sum_{i=1}^j m'_i$, for $j=1,\ldots,k'$.
	\end{itemize}
	Then, $\mathbf{W}' \cong \prod_{i=1}^{k'} \mathrm{IG}(m'_i, n'_i)$.
\end{lemma}
\begin{proof}
	This is straightforward.
\end{proof}

\textbf{Type B}: Let $\bf{d}_B = [d_1, d_2, \ldots, d_r] = [\bf{T}_1, \bf{T}_2, \ldots, \bf{T}_m]$. Take $\bf{T}_{i_a} \in I^B_e$, let $\bf{O}_{\bf{d}^B_{i_a}}$ be the associated minimal Richardson orbit of $\bf{O}^B_e$. Take any polarization $P_B \in \mathrm{Pol}(\bf{O}_{\bf{d}^B_{i_a}})$, suppose the induced paritition \eqref{induction} is
\[
    \bf{d}=[\tilde {d}_1 \equiv \tilde {d}_2 \equiv \ldots \equiv \tilde{d}_{2l+1} \equiv 1, \tilde {d}_{2l+2} \equiv \ldots \equiv \tilde {d}_r \equiv 0].
\]
We define the following two sets of integers:
\begin{itemize}
	\item Let $\{ d'_{1} > d'_{2} > \ldots > d'_k\} \subset \{ d_{2l+2}, \ldots, d_r \}$ be the distinguished odd parts. Denote $n_i = \#\{d_j \in \mathbf{d}_B \mid d_j =d'_i \}$ for $i=1, \ldots, k$.
	\item Let $\{ d''_{1} > d''_{2} > \ldots > d''_{k'}\} \subset \{ d_{1}, \ldots, d_{2l+1} \}$ be the distinguished even parts. Denote $n'_i = \#\{d_j \in \mathbf{d}_B \mid d_j =d''_i \}$ for $i=1, \ldots, k'$. Note that $n'_i \equiv 0$.
\end{itemize}
Construct $\mathbf{W}_B$ (resp. $\mathbf{W}'_B$) as in Lemma~\ref{lem:isotropic subspace} (resp. Lemma~\ref{lem:Lagrangian subspace}) via $\{n_1, \ldots, n_k \}$ (resp. $\{n'_1, \ldots, n'_{k'} \}$).

\begin{theorem}\label{Thm:Spal_fiber B}
	With the above notations, the Spaltenstein variety $\mu_{P_B}^{-1}(e)^{\mathrm{red}}$ (with reduced structure) of 
	\[
	     \mu_{P_B}: T^*(\mathrm{SO}_{2n+1}/P_B) \rightarrow \overline{\mathbf{O}}_{\mathbf{d}_{i_a}^{B}} \supset \mathbf{O}^B_e 
	\] 
	is isomorphic to $\mathbf{W}_B \times_{\mathrm{pt}} \mathbf{W}'_B$.
\end{theorem}

Before giving the proof, we need to introduce the Jordan basis. 

The partition $\mathbf{d}_{B}$ determines a Young tableau $\text{Y}(\mathbf{d}_{B}) \subset \mathbb{Z}^2_{>0}$ such that $(i, j) \in \text{Y}(\mathbf{d}_{B}) $ if and only if $1 \leq j \leq r$ and $1 \leq i \leq d_j$. We choose a Jordan basis of $\mathbb{C}^{2n+1}$, $\left\{ x(i,j) \right\}_{(i,j) \in Y(\bf{d}_{B})}$, for $e \in \mathbf{O}^{B}_e$ as follows
\begin{itemize}
    \item $e \cdot x(i,j) = x(i-1, j) $ for $i > 1$, and $e \cdot x(1,j) = 0 $.
    \item $\langle x(i,j), x(p,q)  \rangle \neq 0 $ if and only if $i+p=d_j+1$ and $q=\beta(j)$. Here $\langle -, - \rangle$ is the pairing on $\mathbb{C}^{2n+1}$ and $\beta$ is a permutation of $\{1, \ldots, r\}$ such that $\beta^2=\id$, $d_{\beta(j)}=d_j$, and $\beta(j) \neq j$ if $d_j \equiv 0$. 
\end{itemize}
In the following, we choose a $\beta$ such that $\beta(j) = j$ if $d_j \equiv 1$.

\begin{proof}[Proof of Theorem~\ref{Thm:Spal_fiber B}]
	Let $\bf{d}=[d_1,\ldots,d_r]=[ \bf{T}_1, \ldots, \bf{T}_{i_a-1}, \bf{T}_{i_a}, \bf{T}_{i_a+1}, \ldots, \bf{T}_m]$ be the partition of $\bf{O}^B_e$. Recall that $\bf{d}_{i_a}=[ \bf{T}''_1, \ldots, \bf{T}''_{i_a-1}, \bf{T}^\circ_{i_a}, \bf{T}'_{i_a+1}, \ldots, \bf{T}'_m]$, and $\bf{d}_{i_a}^B$ is the B-collapse of $\bf{d}_{i_a}$. 

	Let $\bf{d}_{i_a}=[\bar{d}_1 \equiv \bar{d}_2 \equiv \ldots \equiv \bar{d}_{2l'+1} \equiv 1, \bar{d}_{2l'+2} \equiv \ldots \equiv \bar{d}_{r} \equiv 0]$ (here $\bar{d}_{r}$ may be zero), and suppose the induced partition from the pseudo-polarization $P_B$ is $\bf{d}=[\tilde {d}_1 \equiv \tilde {d}_2 \equiv \ldots \equiv \tilde{d}_{2l+1} \equiv 1, \tilde {d}_{2l+2} \equiv \ldots \equiv \tilde {d}_r \equiv 0]$. From \cite[\S 5.1]{FRW24}, we have $l \geq l'$, and if $l' > l$, then the additional odd parts in $\bf{d}_{i_a}$ (i.e., those with indices $2l+2, \ldots, 2l'+1$) arise from the collapse process. More precisely, we have $\bar{d}_{2s+1} \geq \bar{d}_{2s+2}+2$ for $s=l+1, \ldots, l'-1$ and $\tilde{d}_{2s}=\bar{d}_{2s}+1$, $\tilde{d}_{2s+1}=\bar{d}_{2s+1}-1$ for $s=l+1, \ldots, l'$.

	Suppose the Levi type of $P_B$ is $(p_1, p_2, \ldots, p_k; q)$ (we have the requirement: $p_1 \leq p_2 \leq \ldots \leq p_k$), from the induction process~\eqref{induction}, we know that $q=2l+1$, $k=\lfloor \frac{d_1}{2} \rfloor$ and 
	\[
		[p_k, \ldots, p_2, p_1]^t = [\tfrac{\tilde{d}_1-1}{2}, \ldots, \tfrac{\tilde{d}_{2l+1}}{2}, \tfrac{\tilde{d}_{2l+2}}{2}, \ldots, \tfrac{\tilde{d}_{r}}{2}].
	\]

	Now, we first construct the isotropic subspace $F_k^\perp \subset \CC^{2n+1}$ of dimension $\sum_i p_i$ in the admissible flag~\eqref{isotropic flag}. Let $\{ d'_{1}, d'_{2}, \ldots, d'_{m}\} \subset \{ d_{2l+2}, \ldots, d_r \}$ (resp. $\{ d''_{1}, d''_{2}, \ldots, d''_{m'}\} \subset \{ d_{1}, \ldots, d_{2l+1} \}$) be the distinguished odd (resp. even) parts, with multiplicities $\{n_1, n_2, \ldots, n_m\}$ (resp. $\{n'_1, n'_2, \ldots, n'_{m'}\}$). Define the modified sequence $\{p'_1, \ldots, p'_k\}$ as follows:
		\begin{enumerate}
			\item For some $d'_\alpha$ with $n_\alpha > 1$, if $i = k-\lfloor \frac{d'_\alpha}{2} \rfloor$, then $p'_i = p_i - \lfloor \frac{n_\alpha}{2} \rfloor$;
			\item For some $d'_\alpha$ with $n_\alpha = 1$, if $i = k-\lfloor \frac{d'_\alpha}{2} \rfloor$ and $d'_\alpha = d_\alpha + 1$, then $p'_i = p_i - 1$;
			\item For some $d''_\beta$, if $i = k- \frac{d''_\beta}{2}$, then $p'_i = p_i - \frac{n'_\beta}{2}$;
			\item Otherwise, $p'_i = p_i$.
		\end{enumerate}

	Let
	\[
		E =  \bigoplus_{1 \leq i \leq k} \bigoplus_{1 \leq j \leq p'_{k+1-i}} \CC x(i,j),
	\]
	and
	\[
		V = V' \oplus V''= \bigoplus_{1 \leq \alpha \leq m} V'_\alpha \oplus \bigoplus_{1 \leq \beta \leq m'} V''_\beta,
	\]
	where (set $i_\alpha = k-\lfloor \tfrac{d'_\alpha}{2} \rfloor$ and $i_\beta = k- \tfrac{d''_\beta}{2}$)
	\[
		V'_\alpha = \bigoplus_{p'_{i_\alpha}+1 \leq j \leq p'_{i_\alpha}+n_\alpha} \CC x(i_\alpha,j), \quad V''_\beta = \bigoplus_{p'_{i_\beta}+1 \leq j \leq p'_{i_\beta}+n'_\beta} \CC x(i_\beta,j).
	\]

	First notice that $\dim V' \equiv \dim V'' \equiv 0$, and there are nondegenerate pairings on $V'$ and $V''$. The one on $V'$ is induced from the pairing on $\CC^{2n+1}$, while the one on $V''$ is induced from the skew-symmetric pairing defined as follows: let $\{e, h, f \}$ be the standard triple such that $f \cdot x(i,j) \in \CC x(i+1,j)$ for $i < d_j$ and $f \cdot x(d_j,j) = 0$, then the pairing on $V''$ is defined by $\langle v, w \rangle = \langle v, f(w) \rangle$ for $v, w \in V''$.  

	Since, $F_k^\perp$ is isotropic, it can be shown that $E \subset F_k^\perp \subset E \oplus V$. Note that $\dim E = \dim F_k^\perp - \tfrac{1}{2}\dim V$. Then, it is sufficient to construct the maximal isotropic subspace $W \subset V$ such that
	\begin{align}\label{eq:dim condition}
		\begin{split}
			& \dim \big(W \cap \bigoplus_{m-j \leq \alpha \leq m} V'_\alpha \big) \geq \big\lfloor \tfrac{\sum_{m-j \leq \alpha \leq m} n_\alpha}{2} \big\rfloor, \\
			& \dim \big(W \cap (V' \oplus \bigoplus_{m'-j' \leq \beta \leq m'} V''_\beta) \big) \geq \tfrac{1}{2} \dim V' + \big\lfloor \tfrac{\sum_{m'-j' \leq \beta \leq m'} n'_\beta}{2} \big\rfloor.
		\end{split}
	\end{align}
	since $e^i (F_k^\perp) \subset F_{k-i}^\perp$. However, $W$ is isotropic, then the inequalities in \eqref{eq:dim condition} are equalities. Hence, we can apply Lemma~\ref{lem:isotropic subspace} and \ref{lem:Lagrangian subspace}. 

	Finally, by the dimension count, we have $\dim (e^i (F_k^\perp)) = \sum_{j=1}^{k-i} p_j = \dim F^\perp_{k-i}$ for $i=1, \ldots, k-1$. Hence, to construct the admissible flag, it is enough to construct $F_k^\perp$. This finishes the proof.
\end{proof}

\textbf{Type C}: Let $\bf{d}_C = [d_1, d_2, \ldots, d_r] = [\bf{T}_1, \bf{T}_2, \ldots, \bf{T}_m]$. Take $\bf{T}_{i_a} \in I^C_e$, let $\bf{O}_{\bf{d}^C_{i_a}}$ be the associated minimal Richardson orbit of $\bf{O}^C_e$. Take any polarization $P_C \in \mathrm{Pol}(\bf{O}_{\bf{d}^C_{i_a}})$, suppose the induced paritition \eqref{structure of induction} is
\[
    \bf{d}=[\tilde {d}_1 \equiv \tilde {d}_2 \equiv \ldots \equiv \tilde {d}_{2l} \equiv 1, \tilde {d}_{2l+1} \equiv \ldots \equiv \tilde {d}_r \equiv 0].
\]
We define the following two sets of integers:
\begin{itemize}
	\item Let $\{ d'_{1} > d'_{2} > \ldots > d'_k\} \subset \{ d_{1}, \ldots, d_{2l} \}$ be the distinguished even parts. Denote $n_i = \#\{d_j \in \mathbf{d}_C \mid d_j =d'_i \}$ for $i=1, \ldots, k$.
	\item Let $\{ d''_{1} > d''_{2} > \ldots > d''_{k'}\} \subset \{ d_{2l+1}, \ldots, d_{r} \}$ be the distinguished odd parts. Denote $n'_i = \#\{d_j \in \mathbf{d}_C \mid d_j =d''_i \}$ for $i=1, \ldots, k'$. Note that $n'_i \equiv 0$.
\end{itemize}
Construct $\mathbf{W}_C$ (resp. $\mathbf{W}'_C$) as in Lemma~\ref{lem:isotropic subspace} (resp. Lemma~\ref{lem:Lagrangian subspace}) via $\{n_1, \ldots, n_k \}$ (resp. $\{n'_1, \ldots, n'_{k'} \}$). Similarly, we have

\begin{theorem}\label{Thm:Spal_fiber C}
	With the above notations, the Spaltenstein variety $\mu_{P_C}^{-1}(e)^{\mathrm{red}}$ (with reduced structure) of 
	\[
	    \mu_{P_C}: T^*(\mathrm{Sp}_{2n}/P_C) \rightarrow \overline{\mathbf{O}}_{\mathbf{d}_{i_a}^{C}} \supset \mathbf{O}^C_e 
	\] 
	is isomorphic to $\mathbf{W}_C \times_{\mathrm{pt}} \mathbf{W}'_C$.
\end{theorem}

\textbf{Type D}: Let $\bf{d}_D = [d_1, d_2, \ldots, d_r] = [\bf{T}_1, \bf{T}_2, \ldots, \bf{T}_m]$. Take $\bf{T}_{i_a} \in I^D_e$, let $\bf{O}_{\bf{d}^D_{i_a}}$ be the associated minimal Richardson orbit of $\bf{O}^D_e$. Take any polarization $P_D \in \mathrm{Pol}(\bf{O}_{\bf{d}^D_{i_a}})$, suppose the induced paritition \eqref{structure of induction} is
\[
    \bf{d}=[\tilde {d}_1 \equiv \tilde {d}_2 \equiv \ldots \equiv \tilde \equiv {d}_{2l} \equiv 1, \tilde {d}_{2l+1} \equiv \ldots \equiv \tilde {d}_r \equiv 0].
\]
We define the following two sets of integers:
\begin{itemize}
	\item Let $\{ d'_{1} > d'_{2} > \ldots > d'_k\} \subset \{ d_{2l+1}, \ldots, d_r \}$ be the distinguished odd parts. Denote $n_i = \#\{d_j \in \mathbf{d}_D \mid d_j =d'_i \}$ for $i=1, \ldots, k$.
	\item Let $\{ d''_{1} > d''_{2} > \ldots > d''_{k'}\} \subset \{ d_{1}, \ldots, d_{2l} \}$ be the distinguished even parts. Denote $n'_i = \#\{d_j \in \mathbf{d}_D \mid d_j =d''_i \}$ for $i=1, \ldots, k'$. Note that $n'_i \equiv 0$.
\end{itemize}
Construct $\mathbf{W}_D$ (resp. $\mathbf{W}'_D$) as in Lemma~\ref{lem:isotropic subspace} (resp. Lemma~\ref{lem:Lagrangian subspace}) via $\{n_1, \ldots, n_k \}$ (resp. $\{n'_1, \ldots, n'_{k'} \}$). Similar as type B case, we have

\begin{theorem}\label{Thm:Spal_fiber D}
	With the above notations, the Spaltenstein variety $\mu_{P_D}^{-1}(e)^{\mathrm{red}}$ (with reduced structure) of 
	\[
	     \mu_{P_D}: T^*(\mathrm{SO}_{2n}/P_D) \rightarrow \overline{\mathbf{O}}_{\mathbf{d}_{i_a}^{D}} \supset \mathbf{O}^D_e 
	\] 
	is isomorphic to $\mathbf{W}_D \times_{\mathrm{pt}} \mathbf{W}'_D$.
\end{theorem}

\begin{example}\label{exm: singular fiber}
    In this example we refer to 1.3.8 of \cite{Yun17}. Consider an algebraic group of adjoint type and the subregular nilpotent orbit $\bf{O}_{subreg}\subset \overline{\bf{O}}_{reg}$. It is the unique nilpotent orbit of codimension two in the nilpotent cone and it is a Richardson orbit. 

    Thus the minimal Richardson orbit of $\bf{O}^{subreg}$ is itself. However, if we consider 
    \[
    \mu_B: T^*(G/B)\rightarrow \overline{\bf{O}}_{reg} \supset \bf{O}_{subreg}
    \]
    then for any $e\in \bf{O}_{subreg}$, the Springer fiber $\mu_B^{-1}(e)$ is a configuration of $\mathbb{P}^1$'s which forms a Dynkin curve, and in particular, it is not smooth in all most all cases.

    Therefore, this example confirms the distinctive property of the minimal Richardson orbit.
\end{example}

\section{E-polynomials of Spaltenstein varieties and duality}\label{Sec: E poly}

In this section, we only consider special nilpotent orbits of types B and C since $G :=\SO_{2n+1}$ and  ${}^LG=\Sp_{2n}$ are Langlands dual to each other.

Denote by $\underline{\cN}_{B/C}^{sp}$ the set of special nilpotent orbits in type B/C. By Springer theorem, there is a bijection map 
    \begin{align} \label{Springer_dual}
        S: \underline{\cN}^{sp}_B \longrightarrow \underline{\cN}^{sp}_C,
    \end{align}
which is called the \emph{Springer dual}. For a special nilpotent orbit $\bf{O}_B$ of type B, we denote its Springer dual by ${}^S\bf{O}_B$. Moreover, the Springer dual is dimension-preserving, then combining with \cite[Proposition 3.1]{FRW24}, we have

\begin{proposition}\label{prop.Spr dual min Richardson}
    If $\bf{O}_{B,R}$ is a minimal Richardson orbit for $\bf{O}_B$. Then ${}^S\bf{O}_{B,R}$ is a minimal Richardson orbit for ${}^S\bf{O}_B$.
\end{proposition}

In terms of partitions, let $\bf{d}_B = [\bf{T}_1, \bf{T}_2, \ldots, \bf{T}_m]$ be the partition of a special orbit $\mathbf{O}_B$. Then, the partition of its Springer dual $\mathbf{O}_C = {}^S\mathbf{O}_B$ is given by
\begin{align}\label{dual_partition}
	\bf{d}_C = [\tilde{\bf{T}}_1, \tilde{\bf{T}}_2, \ldots, \tilde{\bf{T}}_m],
\end{align}
where $\tilde{\bf{T}}_i = \bf{T}_i$ if $\bf{T}_i$ is of type B1, and $\tilde{\bf{T}}_i = \bf{T}_i'$ if $\bf{T}_i$ is of type B2 or B3.

Let $\mathbf{O}_B$ and $\mathbf{O}_C$ be Springer dual special orbits of types B and C. Let $\mathbf{O}_{B,R}$ and $\mathbf{O}_{C,R}$ be Springer dual minimal Richardson orbits. Taking Langlands dual polarizations $P_B$ and $P_C$, i.e., the Levi types~\eqref{levi-type} have the forms:
\[
	\mathfrak{l}_B \cong \mathfrak{gl}_{p_1} \oplus \mathfrak{gl}_{p_2} \oplus \cdots \oplus \mathfrak{gl}_{p_k} \oplus \mathfrak{so}_{2l+1}, \quad \mathfrak{l}_C \cong \mathfrak{gl}_{p_1} \oplus \mathfrak{gl}_{p_2} \oplus \cdots \oplus \mathfrak{gl}_{p_k} \oplus \mathfrak{sp}_{2l}.
\]

By \cite[Proposition 3.1]{FRW24}, we have 
\begin{align*}
    \xymatrix{ 
     & T^*(\SO_{2n+1}/P_{B})  \ar[d]_{\mu_{P_B}} 
     & \stackrel{\text{Langlands dual}}{\leftrightsquigarrow} 
     & T^*(\Sp_{2n} / P_{C}) \ar[d]^{\mu_{P_C}} \\     
     \mathbf{O}_{B} \ar@{^{(}->}[r] 
     & \overline{\mathbf{O}}_{B,R} 
     & \stackrel{\text{Springer dual}}{\leftrightsquigarrow} 
     & \overline{\mathbf{O}}_{C,R} 
     & \mathbf{O}_{C} \ar@{_{(}->}[l]  
    }
\end{align*}
Notice that, by Theorem~\ref{Thm:Spal_fiber B} and \ref{Thm:Spal_fiber C}, the Spaltenstein varieties $\mu_{P_B}^{-1}(e)^{\rm red}$ and $\mu_{P_C}^{-1}(e')^{\rm red}$ are smooth projective varieties. The Hodge-Deligne polynomial of a smooth variety $Z$ is defined as
\begin{align*}
	{\rm E}(Z; u, v)=\sum_{p, q} \sum_{k \geq 0}(-1)^{k} h^{p, q}\left(H_{c}^{k}(Z, \mathbb{C})\right) u^{p} v^{q},
\end{align*}
where $h^{p, q}\left(H_{c}^{k}(Z, \mathbb{C})\right)$ is the dimension of $(p, q)$--th Hodge-Deligne component in the $k$--th cohomology group with compact supports. Moreover, if $f: Z \rightarrow Y$ is a Zariski locally trivial fibration with fiber $F$ over a closed point, then from the exact sequence of mixed Hodge structures, we have ${\rm E}(Z)={\rm E}(Y) \times {\rm E}(F)$.

\begin{thm}\label{thm.duality}
	Under the above condition. Let $e \in \mathbf{O}_{B}$ and $e' \in \mathbf{O}_{C} $. Then, the connected components of $\mu_{P_B}^{-1}(e)^{\rm red}$ and $\mu_{P_C}^{-1}(e')^{\rm red}$ share the same E-polynomial. Furthermore, let $\#\mu_{P_B}$ be the number of connected component of $\mu_{P_B}^{-1}(e)^{\rm red}$, and similar for $\#\mu_{P_C}$. Then we have the following seesaw property
	\begin{align*}
        \#\mu_{P_B} \cdot \#\mu_{P_C} = \#\bar{A}(\mathbf{O}_{B}) = \#\bar{A}(\mathbf{O}_{C}).
    \end{align*}
	Here, $\bar{A}(\mathbf{O})$ is the Lusztig's quotient group of the component group $A(\mathbf{O})$.  
\end{thm}

\begin{proof}
	Let $\bf{d}_B=[d_1, d_2, \ldots, d_r]=[\bf{T}_1, \bf{T}_2, \ldots, \bf{T}_m]$, and suppose the induced partition from the pseudo-polarization $P_B$ is $\bf{d}=[\tilde {d}_1 \equiv \tilde {d}_2 \equiv \ldots \equiv \tilde {d}_{2l+1} \equiv 1, \tilde {d}_{2l+2} \equiv \ldots \equiv \tilde {d}_r \equiv 0]$. Let $n_1,\ldots, n_k$ be the multiplicities of the distinguished odd parts $d_{i_1}, \ldots, d_{i_k}$ in $\{d_{2l+2}, \ldots, d_r\}$, i.e.,
	\[
		[\ldots, d_{i_1} = d_{i_1+1} = \ldots = d_{i_1+n_1-1}, \ldots, d_{i_k} = d_{i_k+1} = \ldots = d_{i_k+n_k-1}, \ldots].
	\]
	Let $\bf{d}_C=[d'_1, d'_2, \ldots, d'_r]$ (here we allow $d'_r$ to be zero) be the partition of $\mathbf{O}_C = {}^S\mathbf{O}_B$. Let $n'_1,\ldots, n'_{k}$ be the multiplicities of the distinguished odd parts in $\{d'_{2l+1}, \ldots, d'_r \}$. By the relation~\eqref{dual_partition}, we have 
	\begin{enumerate}
		\item $n'_j = n_j$ if $n_j$ is even and $d_{i_j}$, $d_{i_j+n_j-1}$ are in blocks of type B1. 
		\item $n'_j = n_j-2$ if $n_j$ is even but $d_{i_j}$, $d_{i_j+n_j-1}$ are in different blocks of type B2 or B3.
		\item $n'_j = n_j-1$ if $n_j$ is odd, note that in this case, either $d_{i_j}$ or $d_{i_j+n_j-1}$ is in the block of type B2 or B3.
	\end{enumerate}

	By Theorem~\ref{Thm:Spal_fiber B} and \ref{Thm:Spal_fiber C}, we have
	\[
		\mu_{P_B}^{-1}(e)^{\rm red} \cong \mathbf{W}_B \times_{\mathrm{pt}} \mathbf{W}'_B, \quad \mu_{P_C}^{-1}(e')^{\rm red} \cong \mathbf{W}_C \times_{\mathrm{pt}} \mathbf{W}'_C.
	\]
	Let $\mathbf{W}^\circ$ denote a connected component of $\mathbf{W}$. It suffices to show that
	\[
		{\rm E}(\mathbf{W}^\circ_B) = {\rm E}(\mathbf{W}'_C), \quad {\rm E}(\mathbf{W}'_B) = {\rm E}(\mathbf{W}^\circ_C).
	\]
	From \cite[Proposition 3.10]{FRW24}, we know 
	\begin{align*}
		{\rm E}(\operatorname{OG}(n,2n+1)) = {\rm E}(\operatorname{IG}(n,2n)) &= \prod_{j=1}^n (q^{j}+1),\\  {\rm E}(\operatorname{OG}(n,2n)) = 2{\rm E}(\operatorname{IG}(n-1,2n-2)) &= 2\prod_{j=1}^{n-1}(q^j+1),
	\end{align*}
	where $q=uv$.

    Recall from Lemma~\ref{lem:isotropic subspace}, $\bf{W}_B \rightarrow \ldots \rightarrow \bf{W}_2 \rightarrow \bf{W}_1$, where $\bf{W}_1 \cong \operatorname{OG}(\lfloor \frac{n_1}{2} \rfloor, n_1)$ and $\bf{W}_2 \rightarrow \bf{W}_1$ is a 
    \[
        \operatorname{OG}(\lfloor \tfrac{n_1+n_2}{2} \rfloor-\lfloor \tfrac{n_1}{2} \rfloor, n_1+n_2-2\lfloor \tfrac{n_1+n_2}{2} \rfloor)
    \]
    fibration.

	Note that $d_{i_1}$ is in the block of type B2. If $n_1$ is odd, then $n_1'=n_1-1$, from the above equation, we know ${\rm E}(\operatorname{OG}( \frac{n_1-1}{2} ,n_1))={\rm E}(\operatorname{IG}( \frac{n_1'}{2} ,n_1'))$. If $n_1$ is even, we arrive at case (2) mentioned above, i.e., $n_1'=n_1-2$. Then, we have ${\rm E}(\operatorname{OG}( \frac{n_1}{2} ,n_1))=2{\rm E}(\operatorname{IG}( \frac{n_1'}{2} ,n_1'))$, here the multiplicity 2 comes from the fact that $\operatorname{OG}(\frac{n_1}{2},n_1)$ has two connected components. 
	
	If both $n_1$ and $n_2$ are odd, then $d_{i_2+n_2-1}$ is in the block of type B2 or B3, and $n_2'=n_2-1$. We have ${\rm E}(\operatorname{OG}( \lfloor\frac{n_1+n_2}{2}\rfloor - \lfloor \frac{n_1}{2} \rfloor, n_1+n_2-2\lfloor \frac{n_1}{2} \rfloor))=2{\rm E}(\operatorname{IG}( \frac{n_2'}{2} ,n_2'))$. 
	
	If both $n_1$ and $n_2$ are even, we arrive at case (2), i.e., $n_2'=n_2-2$. Then, we have ${\rm E}(\operatorname{OG}( \lfloor \frac{n_1+n_2}{2} \rfloor - \lfloor \frac{n_1}{2} \rfloor ,n_1+n_2-2\lfloor \frac{n_1}{2} \rfloor))=2{\rm E}(\operatorname{IG}( \frac{n_2'}{2} ,n_2'))$.

    If $n_1$ is odd and $n_2$ is even, then $d_{i_2}, \ldots, d_{i_2+n_2-1}$ are all in blocks of type B1. In this case, $n_2'=n_2$, and we have ${\rm E}(\operatorname{OG}( \lfloor\frac{n_1+n_2}{2}\rfloor - \lfloor \frac{n_1}{2} \rfloor, n_1+n_2-2\lfloor \frac{n_1}{2} \rfloor))={\rm E}(\operatorname{IG}( \frac{n_2'}{2} ,n_2'))$.

    If $n_1$ is even and $n_2$ is odd, we arrive at case (3) and $d_{i_2}$ is in the block of type B2, $n_2'=n_2-1$. We have ${\rm E}(\operatorname{OG}( \lfloor\frac{n_1+n_2}{2}\rfloor - \lfloor \frac{n_1}{2} \rfloor, n_1+n_2-2\lfloor \frac{n_1}{2} \rfloor))={\rm E}(\operatorname{IG}( \frac{n_2'}{2} ,n_2'))$.
    
    Continuing this process, using the description of $\bf{W}_B$ from Lemma~\ref{lem:isotropic subspace} and the property of E-polynomial, we conclude that ${\rm E}(\mathbf{W}^\circ_B) = {\rm E}(\mathbf{W}'_C)$. Similarly, we can prove ${\rm E}(\mathbf{W}'_B) = {\rm E}(\mathbf{W}^\circ_C)$.

	Finally, from \cite[Proposition 2.7]{FRW24} or \cite[Lemma 2.24]{WWW24}, we have $ \bar{A}(\mathbf{O}_B) \cong \bar{A}({}^S\mathbf{O}_B) = \mathbb{Z}_2^{q}$. Here $q$ equals the number of blocks of type B2 in $\mathbf{d}_B$. On the other hand, the different connected components of $\mu_{P_B}^{-1}(e)^{\rm red}$ and $\mu_{P_C}^{-1}(e')^{\rm red}$ come from the $\operatorname{OG}$ fibration. 

    Since $\bf{O}_B$ is special, there is no type B1* block in $\bf{d}_B$. From the proof of Lemma~\ref{lem:isotropic subspace}, we can analyze each type B2 block and it is easy to show each type B2 block, before $d_{2l+1}$, contributes two connected components of $\bf{W}_C$. To find the number of connected components of $\bf{W}_B$, first note that $\sum_{i=1}^k n_i$ is even, then we can divide them into pieces $\{n_1, \ldots, n_{i_1}\}, \{n_{i_1+1}, \ldots, n_{i_2}\}, \ldots, \{n_{i_{p-1}+1}, \ldots, n_k\}$, such that 
	\[
		\sum_{j=i_{s-1}+1}^{i_s} n_j \equiv 0, \quad \sum_{j=i_{s-1}+1}^{a} n_j \equiv 1 \quad \text{for any }  i_s < a < i_{s+1}.
	\]
	Here we set $i_0=0$. From the proof of Lemma~\ref{lem:isotropic subspace}, we know each piece contributes connected components of $\bf{W}_B$ and from the above analysis, each piece contributes two connected components. Note that the number of pieces equals the number of type B2 blocks after $d_{2l+1}$ plus one in $\bf{d}_B$. Hence, we have $\#\mu_{P_B} \cdot \#\mu_{P_C} = \#\bar{A}(\mathbf{O}_{B}) = \#\bar{A}(\mathbf{O}_{C})$.
\end{proof}

\bibliographystyle{alpha}
	
\bibliography{ref}

\end{document}